\documentclass[twoside, 11pt]{article}
\usepackage{amssymb, amsmath, mathrsfs, amsthm}
\usepackage{tikz,multirow}
\usepackage{graphicx}
\usepackage{color}
\usepackage{geometry}
\usepackage{float}

\usetikzlibrary{decorations.markings,decorations.pathreplacing}
\tikzstyle{vertex}=[circle, draw, inner sep=0pt, minimum size=4pt]
\newcommand{\vertex}{\node[vertex]}

\newcommand{\bd}{\begin{description}}
\newcommand{\ed}{\end{description}}
\newcommand{\bi}{\begin{itemize}}
\newcommand{\ei}{\end{itemize}}
\newcommand{\be}{\begin{enumerate}}
\newcommand{\ee}{\end{enumerate}}
\newcommand{\beqs}{\begin{eqnarray*}}
\newcommand{\eeqs}{\end{eqnarray*}}

\newcommand{\pp}{^{\prime \prime}}

\definecolor{DarkGreen}{rgb}{0.2, 0.6, 0.3}

\newcommand{\xRightarrow}[1]{\overset{#1}{\Longrightarrow}}

\newtheorem{theorem}{Theorem}
\newtheorem{lemma}{Lemma}

\newtheorem{corollary}[theorem]{Corollary}
\newtheorem{case}{Case}

\newtheorem{claim}{Claim}
\newtheorem{subclaim}{Subclaim}[claim]
\newtheorem{fact}{Fact}

\newtheorem{conjecture}{Conjecture}

\begin{document}

\title{\bf Properly colored $C_{4}$'s in edge-colored graphs}
\author{Chuandong Xu\footnote{School of Mathematics and Statistics, Xidian University, Xi'an, Shaanxi 710126, P.R.~China. {\tt xuchuandong @xidian.edu.cn}}, Colton Magnant\footnote{Department of Mathematics, Clayton State University, Morrow, GA 30260, USA.  {\tt dr.colton.magnant@gmail. com}} \footnote{Academy of Plateau Science and Sustainability, Xining, Qinghai 810008, China}, Shenggui Zhang\footnote{Department of Applied Mathematics, Northwestern Polytechnical University, Xi'an, Shaanxi 710129, P.R.~China. {\tt sgzhang@nwpu.edu.cn}} \footnote{Xi'an-Budapest Joint Center for Combinatorics, Northwestern Polytechnical University, Xi'an, Shaanxi 710129, P.R.~China.}}
\maketitle

\begin{abstract}
When many colors appear in edge-colored graphs, it is only natural to expect rainbow subgraphs to appear. This anti-Ramsey problem has been studied thoroughly and yet there remain many gaps in the literature. Expanding upon classical and recent results forcing rainbow triangles to appear, we consider similar conditions which force the existence of a properly colored copy of $C_{4}$.
\end{abstract}

\section{Introduction}

We consider colorings of the edges of graphs so let $G$ be an edge-colored graph. The number of edges in $G$ and the number of colors appearing on the edges of $G$ are denoted by $e(G)$ and $c(G)$, respectively. A subgraph of $G$ is called {\it rainbow} (or {\it properly colored}), if the colors of all its edges (respectively adjacent edges) are distinct. For other standard notation and terminology, we refer the reader to \cite{CLZ11}.

In \cite{ESS1973}, Erd{\H{o}}s, Simonovits, and S{\'o}s posed the following conjecture.

\begin{conjecture}[Erd{\H{o}}s, Simonovits, S{\'o}s \cite{ESS1973}] \label{essconj}
Let $G$ be an edge-colored graph with order $n$. For all $n \geq k \geq 3$, if
$$
c(G) \geq \left( \frac{k - 2}{2} + \frac{1}{k - 1} \right) n + O(1),
$$
then $G$ contains a rainbow cycle $C_{k}$.
\end{conjecture}

This conjecture was proven in $2005$ by Montellano-Ballesteros and Neumann-Lara \cite{MR2190794} with a simplified proof provided by Choi in \cite{Choi}, and the precise value is well known for the triangle ($C_3$), as seen in the following result.

\begin{theorem}[Erd\H{o}s, Simonovits, S\'os \cite{ESS1973}]
	Let G be an edge-colored complete graph on $n \geq 3$ vertices. If $c(G) \geq n$, then $G$ contains a rainbow $C_3$.
\end{theorem}

More generally, the same conclusion is possible for non-complete graphs by counting both the number of edges and the number of colors. Note that if $G$ is complete, the following result is precisely the previous result.

\begin{theorem}[Li, Ning, Xu, Zhang \cite{LNXZ2014}]\label{Thm:C3inG}
	Let $G$ be an edge-colored graph on $n$ vertices for $n\geq3$. If $e(G)+c(G) \geq n(n+1)/2$, then $G$ contains a rainbow $C_3$.
\end{theorem}

The tight lower bound $n$ in Theorem 1 is also called the {\it rainbow number} (or $n-1$ the {\it anti-Ramsey number}) of $C_3$ in some literature. Thus Theorem 2 is a generalization of the rainbow number of $C_3$ from complete graphs to non-complete graphs. Xu et.~al.~\cite{XHWZ2016} studied the existence of rainbow cliques under $e(G)+c(G)$ condition, which also generalized the rainbow number (anti-Ramsey number) of cliques. For more results on anti-Ramsey problems, see \cite{CLT2009,ESS1973,KL2008,MN2002,Sch2004} with a dynamic survey in \cite{FMO2014}.

In Theorem \ref{Thm:C3inG}, when the assumption almost holds but not quite, it is natural to ask what graphs satisfy the assumption but contain no rainbow triangle. This question is answered in the following result.

\begin{theorem}[Fujita, Ning, Xu, Zhang \cite{FNXZ2017+}]\label{Thm:C3inGstruc}
	Let $G$ be an edge-colored graph of order $n$. If $e(G)+c(G)=n(n+1)/2-1$ and $G$ contains no rainbow triangle, then $G$ belongs to $\mathcal{G}_0$ which can be characterized in the following way:
	\begin{enumerate}
		\item $K_1\in \mathcal{G}_0$;
		\item For every $G \in \mathcal{G}_0$ of order $n \geq 2$, $c(G) = n-1$ and there is a partition $V(G)=V_1\cup V_2$ such that $G[V_1,V_2]$ is monochromatic and $G[V_i] \in \mathcal{G}_0$ for $i = 1,2$.
	\end{enumerate}
\end{theorem}

In this work, we extend Theorems~\ref{Thm:C3inG} and~\ref{Thm:C3inGstruc} to consider a properly colored cycle $C_{4}$ in place of the rainbow triangle. Our first main result count only the number of colors appearing in a complete graph $K_{n}$.

\begin{theorem}\label{Thm:nP1C4}\label{thm_C4inKn}
If $G$ is an edge-colored complete graph $K_n$ using at least $n+1$ colors, then $G$ contains a properly colored $C_4$.
\end{theorem}

In the process of preparing this paper, Li et al.\ \cite{LBZ2019} proved that under the same condition $G$ contains either a properly colored $K_4-e$ or a properly colored $C_5$ together with a chord which both contains a properly colored $C_4$. We present a proof of the above theorem in Section $3$ for completeness.

\begin{theorem}\label{Thm:Classify}\label{thm_C4inKnstruc}
If $G$ is an edge-colored complete graph $K_{n}$ using $n$ colors which contains no properly colored $C_{4}$, then $G$ contains either
\bd
\item{(1)} a vertex with all incident edges in a single color,
\item{(2)} a rainbow triangle $uvw$, say with $uv$ in red, $vw$ in blue, and $uw$ in green, with all other incident edges of $u$ being red, all other incident edges of $v$ being blue, and all other incident edges of $w$ being green, or
\item{(3)} a vertex with all incident edges in distinct colors while all other edges of $G$ have a single (distinct) color.
\ed
\end{theorem}

\begin{figure}[bht]
	\centering
	\begin{tikzpicture}[x=0.8cm, y=0.8cm]
	\vertex (x1) at (0,0) [label=left:{$v$}]{};
	\vertex (y1) at (6,0) [label=below:{$u$}]{};
	\vertex (y2) at (5,1) [label=above:{$v$}] {};
	\vertex (y3) at (5,-1) [label=below:{$w$}] {};
	\vertex (z1) at (11,0) [label=left:{$v$}]{};
	
	\path (x1) edge (1.6,1.1) edge (1.4,0) edge (1.6,-1.1);
	\path (y1) edge (y2) edge (7.4,0.3) edge (7.4,-0.3);
	\path[dotted] (y2) edge (y3) edge (7.6,0.7) edge (7.8,1.2);
	\path[dashed] (y3) edge (y1) edge (7.6,-0.7) edge (7.8,-1.2);
	\path[dotted] (z1) edge (12.6,1.1);
	\path[dashed] (z1) edge (12.4,0.3);
	\path[dashdotted] (z1) edge (12.4,-0.3);
	\path (z1) edge (12.6,-1.1);

	\node (q1) at (2,-1.5) [label=below:{$G-v$}]{};
	\node (q2) at (8,-1.5) [label=below:{$G-\{u,v,w\}$}]{};
	\node (q3) at (13,-1.5) [label=below:{$G-v$}]{};
	
	\draw (2,0) ellipse (1 and 1.5) node {};
	\draw (8,0) ellipse (1 and 1.5) node {};
	\draw (13,0) ellipse (1 and 1.5) node {};
	
	\end{tikzpicture}
	\caption{The possible structures of the edge-colorings of $K_{n}$ with $n$ colors which contains no properly colored $C_{4}$}.
	\label{Fig:C4inKnstruc}
\end{figure}
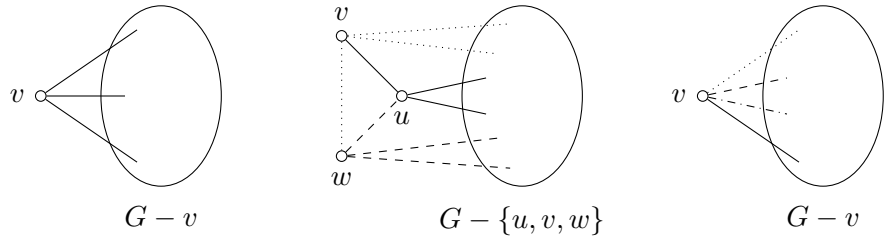

Then, we consider conditions on the number of edges plus the number of colors in a colored graph $G$.

\begin{theorem}\label{thm_C4inG}
Let $G$ be an edge-colored graph with order $n$. If $e(G)+c(G)\geq \frac{n(n+1)}{2}+1$, then $G$ contains a properly colored $C_4$. Moreover, if $e(G)+c(G)=\frac{n(n+1)}{2}$ and $G$ contains no properly colored $C_{4}$, then $G$ is complete and thus has a layered structure characterized by Theorem \ref{thm_C4inKnstruc}.
\end{theorem}

Given an edge-colored graph $G$ with a vertex $v \in V(G)$, let $d^{c}(v)$ denote the color degree of $v$, that is, the number of colors appearing on edges incident to $v$. Li et.~al.~\cite{LNXZ2014} showed that Theorem \ref{Thm:C3inG} implies: $\sum_v d^c_G(v)\geq n(n+1)/2$ is sufficient for the existence of a rainbow $C_3$. This result confirmed (also by Li \cite{Li2013}, independently) a conjecture due to Li and Wang \cite{LW2012} which states that $d^c_G(v)\geq (n+1)/2$ for all $v\in (G)$ is sufficient. In section \ref{Sec:relation} we discussed the relationship between $e(G)+c(G)$ and $\sum_v d^c_G(v)$. As a corollary, we can deduce the following result.

\begin{theorem}\label{Thm:DegSumforC4}
Let $G$ be an edge-colored graph with order $n$ with $n \geq 4$. If $\sum_v d^c(v)\geq \frac{n(n+1)}{2}+1$, then $G$ contains a properly colored $C_4$.  Moreover, if $\sum_v d^c(v)= \frac{n(n+1)}{2}$ and $G$ contains no properly colored $C_4$, then $G$ is complete and has a layered structure characterized by Theorem~\ref{thm_C4inKnstruc} with its center part a rainbow triangle.
\end{theorem}

In the rest of this paper, some lemmas are given in Section 2, and the proofs of Theorems 4--7 are postponed to Sections 3--6, respectively. In the last section, we discuss the structure of edge-colored complete graphs containing no properly colored $C_4$. This is motivated by the study of edge-colored complete graphs containing no rainbow $C_3$, called Gallai-Ramsey problems (see \cite{FMO2014} for a survey).

\section{Preliminaries}

We first state a helpful lemma on the structure of edge-colored complete graphs containing no properly colored $C_{4}$.

\begin{lemma}[Magnant, Martin, Salehi Nowbandegani \cite{MMS2018}]\label{Lemma:Dominating}
Let $G$ be an edge-colored complete graph containing no properly colored $C_{4}$ and for each color $i$, let $G^{i}$ be the subgraph of $G$ containing only edges of color $i$ (not including isolated vertices). Then each $G^{i}$ contains a dominating vertex.
\end{lemma}

In particular, this means that each $G^{i}$ is a threshold graph (for convenience, we allow these threshold graphs to contain isolated vertices). See Figure~\ref{Fig:Threshold}. For each $G^{i}$, there is a sequence of vertices $\{v_{1}, v_{2}, \dots, v_{t}\}$ such that the removal of $\{v_{1}, v_{2}, \dots, v_{r}\}$ for some $r \leq t$ along with any new isolated vertices, leaves behind another threshold graph. Define the \emph{spine} of $G^{i}$ to be the sequence of vertices $\{v_{1}, v_{2}, \dots, v_{t}\}$ and define the \emph{head} of the spine to be $v_{1}$. Note that the sequence may not be unique, so when there is more than one choice for a head, select one arbitrarily from the available choices. After the removal of $v_{t}$, there remain no edges of color $i$ but there must be at least one remaining vertex, called a \emph{tail} or $i$-tail, that was adjacent to $v_{t}$ (and therefore all other vertices of the spine) in color $i$. Again the choice of the tail may not be unique but we often arbitrarily select one from the available choices. In particular, each tail is adjacent to every vertex of the spine but no other vertices in $G^{i}$. Finally, for an edge $e = uv$, let $c(e)$ or $c(uv)$ denote the color of $e$.

\begin{figure}[bht]
	\centering
	\begin{tikzpicture}[x=0.8cm, y=0.8cm]
	\vertex (v1) at (0,4) []{}; \vertex (v2) at (0,3) []{};
	\vertex (v3) at (0,2) []{}; \vertex (v4) at (0,0) []{};
	\vertex (w1) at (3,5) []{}; \vertex (w2) at (3,4) []{};
	\vertex (w3) at (3,3) []{}; \vertex (w4) at (3,2) []{};
	\vertex (w5) at (3,0) []{}; \vertex (w6) at (1.5,-1) []{};
	
	\vertex (x1) at (8,4) []{}; \vertex (x2) at (8,3) []{};
	\vertex (x3) at (8,2) []{}; \vertex (x4) at (8,0) []{};
	\vertex (y1) at (11,5) []{}; \vertex (y2) at (11,4) []{};
	\vertex (y3) at (11,3) []{}; \vertex (y4) at (11,2) []{};
	\vertex (y5) at (11,0) []{}; \vertex (y6) at (9.5,-1) []{};
	
	\path (v1) edge (w2) edge (w3) edge (w4) edge (w5) edge (w6);
	\path (v2) edge (w3) edge (w4) edge (w5) edge (w6);
	\path (v3) edge (w3) edge (w4) edge (w5) edge (w6);
	\path (v4) edge (w5) edge (w6);
	
	\path (y1) edge (x1) edge (x2) edge (x3) edge (x4) edge (y6);
	\path (y2) edge (x2) edge (x3) edge (x4) edge (y6);
	\path (y3) edge (x3) edge (x4) edge (y6);
	\path (y4) edge (x3) edge (x4) edge (y6);
	\path (y5) edge (y6);
	
	\path[dotted] (-1,4.5) edge (5,4.5);
	\path[dotted] (-1,3.5) edge (5,3.5);
	\path[dotted] (-1,1.5) edge (5,1.5);
	\path[dotted] (-1,-1.5) edge (5,-1.5);
	
	\node (q1) at (0,1) [label=center:{$\vdots$}]{};
	\node (q2) at (3,1) [label=center:{$\vdots$}]{};
	\node (q3) at (8,1) [label=center:{$\vdots$}]{};
	\node (q4) at (11,1) [label=center:{$\vdots$}]{};
	
	\draw (0,2) ellipse (0.5 and 2.5) node {};
	\draw (11,2.5) ellipse (0.5 and 3) node {};
	\draw[decoration={brace,mirror,raise=5pt},decorate]
	(w6) -- (w5);
	
	\node (p4) at (-1,3.5) [label=above:{head}]{};
	\node (p5) at (3,-1.3) [label=above:{tails}]{};
	\node (p6) at (0,-1.5) [label=below:{spine}]{};
	\node (p7) at (3,-1.5) [label=below:{ribs}]{};
	
	\end{tikzpicture}
	\caption{The structure of a threshold graph (allow isolate vertices) and its complement.}
	\label{Fig:Threshold}
\end{figure}
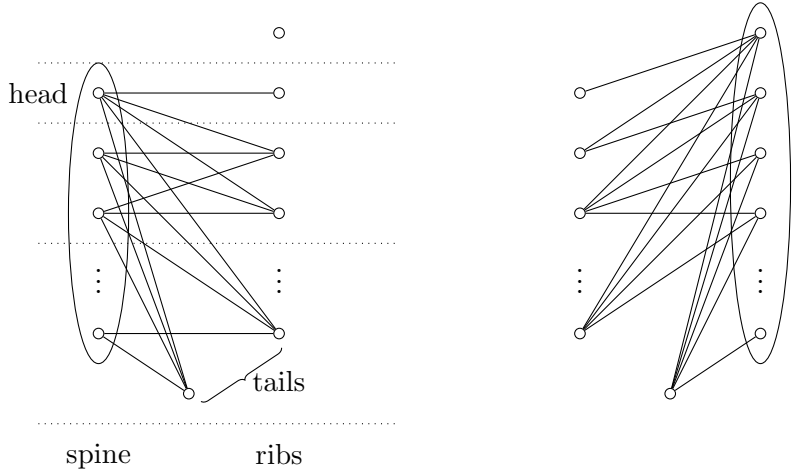

Also note that threshold graphs have been characterized in the following result.

\begin{lemma}[Chv\'atal, Hammer \cite{CH1973}]\label{Lemma:Forbid}
A graph is a threshold graph if and only if it contains no $2K_2$, $P_{4}$ or $C_{4}$ as an induced subgraph.
\end{lemma}

Let $G$ be an edge-colored graph. For each $v\in V(G)$, the {\it colored saturated degree} of $v$ is $d^s(v):=e(G)-e(G-v)$ (introduced by Li et al.~\cite{LNXZ2014}). The edge induced subgraph of $G$ by edges with color $i$ is denoted by $G^i$. We say that color $i$ contributes $0$ or $1$ or $2$ to $\sum_v d^s(v)$ if $G^i$ contains a $2K_2$ or a $C_3$, or $G^i$ is a star $K_{1,t}$ for $t\geq 2$, or $G^i$ is a single edge, respectively.

\begin{lemma}[Li, Ning, Xu, Zhang \cite{LNXZ2014}]\label{lem_CSdegree}
	Let $G$ be an edge-colored graph, then $\sum_v d^s(v)\leq 2c(G)$. The equality holds if and only if $G$ is rainbow.
\end{lemma}

The following version of this lemma can be found in F\"uredi and Simonovits \cite{FS2013} (Theorem 2.22).
\begin{lemma}[K\H{o}vari, S\'os, Tur\'an \cite{KST1954}]\label{lem_Kovari1954}
	 Let $K_{a,b}$ denote the complete bipartite graph with $a$ and $b$ vertices in its two parts. Then
	\[
		ex(n,K_{a,b})\leq \frac{1}{2}\sqrt[a]{b-1}\cdot n^{2-1/a}+\frac{a-1}{2}n.
	\]
\end{lemma}

Since $K_{2,2}\cong C_4$, we have $ex(n,C_4)\leq \frac{1}{2}n^{3/2}+\frac{1}{2}n.$

\begin{lemma}[Erd\H{o}s, Gallai\cite{EG1959}]\label{lem_matching}
	\[
	ex(n,kK_2)=\max \left\{
	\binom{2k-1}{2},\binom{k-1}{2}+(k-1)(n-k+1)
	\right\}.
	\]
	The extremal graph is $K_{2k-1}\cup \bar{K}_{n-2k+1}$, or $K_{k-1}+ \bar{K}_{n-k+1}$ (depending on the value of $ex(n,kK_2))$.
\end{lemma}

\section{Proof of Theorem~\ref{Thm:nP1C4}}

\begin{proof}
It is trivial when $n=4$. Thus we may assume that $n\geq 5$ and $c(G)=n+1$. Suppose the claim does not hold in general and let $G\cong K_n$ be a counterexample with smallest number of vertices.
	
If there exists a vertex $u\in V(G)$ such that $d^s(u)\leq 1$, then
$$
c(G-u) = c(G)-d^s(u) \geq n=|V(G-u)|+1.
$$
Since $G$ is a minimum counter example, there is a properly colored $C_4$ in $G-u$, and therefore also a properly colored $C_4$ in $G$, a contradiction. We may therefore assume that $d^s(v)\geq2$ for every $v\in V(G)$.
	
It follows from Lemma \ref{lem_CSdegree} that
$$
2n\leq \sum_v d^s(v)\leq 2c(G)=2n+2.
$$

\setcounter{case}{0}
\begin{case}
$\sum_v d^s(v)=2n+2$.
\end{case}

By Lemma~\ref{lem_CSdegree}, $G$ is rainbow. Then every $C_4$ in $G$ is a properly colored $C_4$, a contradiction.

\begin{case}
$\sum_v d^s(v)=2n+1$, i.e., exactly $1$ color contributes $1$ and the other $n$ colors contribute $2$ to $\sum_v d^s(v)$.
\end{case}

Let $H$ be the subgraph induced on the $n$ edges with colors which contribute $2$ to $\sum_v d^s(v)$. It follows from the fact that $n\geq 5$ and $ex(n,2K_2)=n-1$ (Lemma \ref{lem_matching}) that $H$ contains a (rainbow) $2K_2$. Each $C_4$ containing the edges of this $2K_2$ is a properly colored $C_4$ in $G$, a contradiction.

\begin{case}
$\sum_v d^s(v)=2n$, and exactly $1$ color contributes $0$ while the other $n$ colors contribute $2$ to $\sum_v d^s(v)$.
\end{case}

The same argument as in the previous case provides the desired result for this case.
	
\begin{case}
$\sum_v d^s(v)=2n$, and exactly $2$ colors contribute $1$ and the other $n-1$ colors contribute $2$ to $\sum_v d^s(v)$.
\end{case}

Let $H$ be the subgraph induced by the $n-1$ edges with colors which each contribute $2$ to $\sum_v d^s(v)$. Since $n\geq 5$, it follows from Lemma~\ref{lem_matching} that the only case when $H$ contains no $2K_2$ is $H\cong K_{1,n-1}$. Since $G$ is complete, the remaining edges with other colors induce a complete graph with order $n-1\geq4$. But there remain only $2$ colors and each color class of these two colors induces a star, a contradiction, completing the proof of the theorem.
\end{proof}

\section{Proof of Theorem~\ref{Thm:Classify}}

Recall the statement of Theorem~\ref{Thm:Classify}. If $G$ is an edge-coloring of the complete graph $K_{n}$ using $n$ colors which contains no properly colored $C_{4}$, then $G$ contains either
\bd
\item{(1)} a vertex with all incident edges in a single color,
\item{(2)} a rainbow triangle $uvw$, say with $uv$ in red, $vw$ in blue, and $uw$ in green, with all other incident edges of $u$ being red, all other incident edges of $v$ being blue, and all other incident edges of $w$ being green, or
\item{(3)} a vertex with all incident edges in distinct colors while all other edges of $G$ have a single (distinct) color.
\ed

\begin{proof}
Let $G$ be an edge-colored complete graph $K_{n}$ using $n$ colors and suppose that $G$ contains no properly colored $C_{4}$. For each vertex $v \in V(G)$, there is a color $c$ such that the graph induced on $V(G) \setminus \{v\}$ contains no edges of color $c$ since otherwise $V(G) \setminus \{v\}$ is an edge-coloring of $K_{n - 1}$ using $n$ colors, which contains a properly colored $C_{4}$ by Theorem~\ref{Thm:nP1C4}. This means $d^s(v)\geq 1$.

We may assume that $G$ contains no vertex with all incident edges in a single color since $G$ would then have Structure~(1).

\begin{claim}\label{Claim:NonHead}
There exists a vertex that is not a head of a spine.
\end{claim}

\begin{proof}
Suppose not, so each vertex is the head of some spine. Note that since there may be more than one choice for the head of a spine, if there is an alternative selection for a head, then we could switch to a different choice and complete the proof of this claim. Thus, each head is unique so there are $n$ colors and $n$ heads. This means that every vertex $u$, say the head of blue, must have at least one neighbor by a blue edge, say to $v$, such that $v$ has no other incident blue edges. Call the edge $uv$ a \emph{lonely blue edge} and call the vertex $v$ a \emph{lonely blue vertex}. Since there may exist only one lonely edge for each head, we select a single lonely edge for each head.

Create an auxiliary graph $H$ on the same vertex set with only the set of all lonely edges. Direct each lonely edge toward its corresponding lonely vertex. This forms a directed graph with out degree equal to $1$ at each vertex. There must therefore exist a directed cycle, which corresponds to a rainbow cycle in $G$. Since $G$ is complete, if this rainbow cycle is even, we may construct a rainbow $C_{4}$ by simply considering chords of a larger rainbow cycle, so this rainbow cycle must be odd.

Suppose first that the rainbow cycle in $G$ corresponding to the directed cycle in $H$ is a $C_{2k + 1}$ where $k \geq 2$. Label the vertices of this cycle as $\{v_{1}, v_{2}, \dots, v_{2k + 1}\}$ with edges of the form $v_{i}v_{i + 1}$ where indices are taken modulo $2k + 1$ so $v_{2k + 2} = v_{1}$. Recall that all edges of this cycle are lonely.

We will now show that there is a rainbow $C_{5}$ in $G$ where all but at most one of the edges are lonely. If $k = 2$, this goal is complete so suppose $k \geq 3$. The chord $v_{1}v_{4}$ cannot make a properly colored $C_{4}$ on $v_{1}v_{2}v_{3}v_{4}v_{1}$ so $v_{1}v_{4}$ must have the same color as either $v_{1}v_{2}$ or $v_{3}v_{4}$. This means that $v_{1}v_{4}v_{5}, \dots, v_{2k + 1}v_{1}$ is a smaller rainbow cycle than the original rainbow cycle but the edge $v_{1}v_{4}$ is not lonely. This process can be repeated using chords $v_{1}v_{2t}$ for $3 \leq t \leq k - 1$ to obtain a rainbow $C_{5}$ with all lonely edges except the edge $v_{1}v_{2k - 2}$.

Relabel the vertices of this rainbow $C_{5}$ as $u_{1}u_{2}u_{3}u_{4}u_{5}u_{1}$ where the lonely edges are $u_{i}u_{i + 1}$ in color $i$ for $1 \leq i \leq 4$ and the potentially non-lonely edge is $u_{5}u_{1}$ in color $5$. Without loss of generality, suppose that $u_{i + 1}$ is the lonely vertex corresponding to the lonely edge $u_{i}u_{i + 1}$. This means that $u_{i + 1}$ has no other incident edges in color $i$ so, in order to avoid a properly colored $C_{4}$, the edge $u_{i + 1}u_{i + 3}$ must have color $i + 3$ for $1 \leq i \leq 4$ where indices are taken modulo $5$. Then the $4$-cycle $u_{1}u_{5}u_{2}u_{4}u_{1}$ is a rainbow (and therefore properly colored) $C_{4}$. This means that there can be no rainbow $C_{5}$ with almost all lonely edges, meaning that there can be no lonely cycle of length at least $4$. Therefore, all lonely cycles must be triangles.

\begin{subclaim}\label{SubClaim:No2-Factor}
The set of lonely edges does not induce a rainbow triangle $2$-factor.
\end{subclaim}

\begin{proof}
Suppose not, so suppose the set of lonely edges induces a rainbow triangle $2$-factor. Let $T = uvw$ be one rainbow triangle of lonely edges and let $T' = xyz$ be another rainbow triangle of lonely edges. Let the edges $uv$, $vw$, $uw$, $xy$, $yz$, and $xz$ have colors $1, 2, 3, 4, 5, 6$ respectively. Also suppose each vertex $v, w, u, y, z, x$ is lonely in color $1, 2, 3, 4, 5, 6$, respectively.

Since $c(uv)=1,c(xy)=4$, $v$ is lonely in $1$ and $y$ is lonely in $4$, it follows from the assumption that $uvyxu$ cannot be properly colored that $c(ux)$ must be $1$ or $4$.

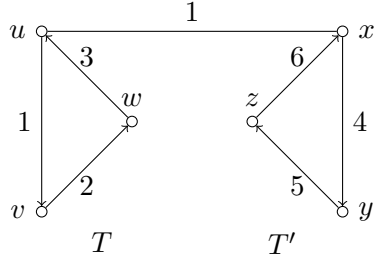
\begin{figure}[bht]
	\centering
	\begin{tikzpicture}[x=0.8cm, y=0.8cm]
	\vertex (u) at (0,3) [label=left:{$u$}]{};
	\vertex (v) at (0,0) [label=left:{$v$}]{};
	\vertex (w) at (1.5,1.5) [label=above:{$w$}] {};
	\vertex (x) at (5,3) [label=right:{$x$}] {};
	\vertex (y) at (5,0) [label=right:{$y$}] {};
	\vertex (z) at (3.5,1.5) [label=above:{$z$}] {};
	
	\path[->]
		(u) edge node[left]{$1$} (v)
		(v) edge node[below]{$2$} (w)
		(w) edge node[above]{$3$} (u)
		(x) edge node[right]{$4$} (y)
		(y) edge node[below]{$5$} (z)
		(z) edge node[above]{$6$} (x)
	;
	\path (u) edge node[above]{$1$} (x);
	
	\node (p1) at (1,0) [label=below:{$T$}]{};
	\node (p2) at (4,0) [label=below:{$T^\prime$}]{};
	\end{tikzpicture}
	\caption{Two rainbow triangles in a triangle 2-factor induced by lonely edges.}
	\label{Fig:Triangle2Factor}
\end{figure}

Up to symmetry, we can now assume that $c(ux)=1$ (see Figure \ref{Fig:Triangle2Factor}). The colors of those edges between $T$ and $T'$ can be determined in the following way, where ``$\xRightarrow{uxyzu}$'' means ``it follows from the assumption that $uxyzu$ is not a properly colored $C_4$ that''.
\begin{table}[H]
	\centering
	\begin{tabular}{ccccc}
	$c(ux)=1$ & $\xRightarrow{uxyzu}$ & $c(uz)=1$ & $\xRightarrow{uzxyu}$ & $c(uy)=1$ \cr
	$\Downarrow {vwuxv}$ & & & & \cr
	$c(vx)=2$ & $\xRightarrow{vxyzv}$ & $c(vz)=2$ & $\xRightarrow{vzxyv}$ & $c(vy)=2$ \cr
	$\Downarrow {wuvxw}$ & & & & \cr
	$c(wx)=3$ & $\xRightarrow{wxyzw}$ & $c(wz)=3$ & $\xRightarrow{wzxyw}$ & $c(wy)=3$. \cr
	\end{tabular}
\end{table}

With all edges between $T$ and $T'$ having the non-lonely color incident to the corresponding vertex in $T$, we may construct an auxiliary digraph $D$ with a vertex for each triangle and a directed edge from $T$ to $T'$ when all edges from $T$ to $T'$ have the colors of $T$. Since each pair of triangles is connected by a directed edge, $D$ is a tournament.

If $D$ is not transitive, then there exists a directed triangle, say $TT'T\pp$. Let $T$ and $T'$ be colored as above and let $T\pp$ be $abc$ with $ab, bc, ac$ having colors $7, 8, 9$ respectively. By the construction of $D$, we also have all edges from $x, y, z$ to $T\pp$ having colors $4, 5, 6$ respectively and all edges from $a, b, c$ to $T$ having colors $7, 8, 9$ respectively. Then the cycle $uxabu$ is rainbow (and therefore properly colored) for a contradiction. This means that $D$ must be transitive. If we let $T$ be the source of the transitive tournament $D$, then $T$ is a rainbow triangle in $G$ with all edges from $u, v, w$ having colors $1, 2, 3$ respectively, and we have the desired result with Structure~(2). Thus, we have shown that the set of lonely edges does not induce a rainbow triangle $2$-factor.
\end{proof}

By Subclaim~\ref{SubClaim:No2-Factor}, there exists a triangle with an extra pendant edge among the lonely edges. Let $T = xyz$ be the triangle where $xy, yz, xz$ have colors $1, 2, 3$ where $y, z, x$ are lonely in colors $1, 2, 3$ respectively. Let $ux$ be the additional edge in color $4$ with $x$ being lonely in color $4$.

Since $z$ is lonely in color $2$, the edge $uz$ must have color $4$ as otherwise $uxyzu$ would be a properly colored $C_{4}$. Since $y$ is lonely in color $1$, the edge $uy$ must also have color $4$ as otherwise $uzxyu$ would be a properly colored $C_{4}$. Notice that this argument never used the fact that $u$ was the head of color $4$ or that $x$ was lonely in color $4$. This means that every vertex in $G \setminus T$ with an edge to the triangle in a color other than $1, 2, 3$ must have all edges in the same color to the triangle. Let $A$ be the set of vertices each with all one color on edges to $T$ and note that $A$ is not empty since $u \in A$.

Suppose there is a vertex $w \in G \setminus (A \cup T)$. By the previous argument, this means that $w$ must have only edges of colors $1, 2, 3$ to $T$. Note that since $x$ is lonely in color $3$, the $c(wx) \in \{1, 2\}$. First suppose $c(wx) = 2$. Then since $z$ is lonely in color $2$, the cycle $wxyzw$ must be properly colored for a contradiction. This means that $c(wx) = 1$ and symmetrically, $c(wy) = 2$ and $c(wz) = 3$. Let $B$ be the set of vertices each with the same colors as $w$ to $T$ and note that $B$ may be empty (see Figure \ref{Fig:TABstructure}).

\begin{figure}[bht]
	\centering
	\begin{tikzpicture}[x=0.8cm, y=0.8cm]
	\vertex (x) at (0,5) [label=left:{$x$}]{};
	\vertex (y) at (0,3) [label=left:{$y$}]{};
	\vertex (z) at (1.5,4) [label=above:{$z$}]{};
	\vertex (u) at (4,4.5) [label=above:{$u$}]{};
	\vertex (w) at (1,1) [label=below:{$w$}]{};
	
	\path[->]
		(x) edge node[left]{$1$} (y)
		(y) edge node[above]{$2$} (z)
		(z) edge node[right]{$3$} (x)
		(u) edge node[above]{$4$} (x)
	;
	\path
		(u) edge node[below]{$4$} (y)
		(u) edge node[above]{$4$} (z)
		(w) edge node[right]{$1$} (x)
		(w) edge node[left]{$2$} (y)
		(w) edge node[right]{$3$} (z)
	;
	\path (u) edge node[right]{$4$} (w);
	
	\node (p1) at (-1,3.5) [label=center:{$T$}]{};
	\node (p2) at (5,4) [label=right:{$A$}]{};
	\node (p3) at (4,1) [label=right:{$B$}]{};
	
	\draw (2,1) ellipse (2 and 0.8) node {};
	\draw (4,4) ellipse (0.8 and 1.5) node {};	
	\end{tikzpicture}
	\caption{The structure of $G$ with partition $V(G)=T\cup A\cup B$.}
	\label{Fig:TABstructure}
\end{figure}
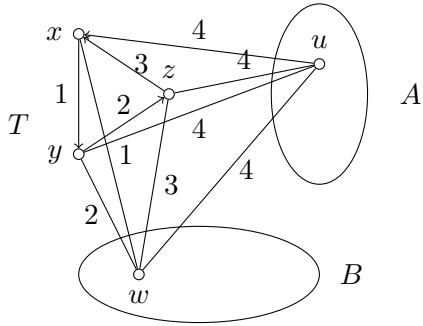

Let $u$ and $v$ be two vertices in $A$ (if two such vertices exist) and suppose $u$ and $v$ have colors $4$ and $5$ respectively on all edges to $T$. To avoid a properly colored $C_{4}$, the edge $uv$ must have either color $4$ or $5$.

Let $u \in A$ and $w \in B$ say with $u$ having color $4$ on edges to $T$. To avoid a properly colored $C_{4}$ in a variety of places using vertices of $T$, the edge $uw$ must have color $4$. This means that each vertex in $A$ has all one color on all edges to $B \cup T$.

An important claim can be deduced here. Here we need the definition of dependence property. A subset $Z$ of $V(G)$ is said to have \emph{dependence property} if each vertex in $Z$ has all one color (assign this color to $v$) on all edges to $G\setminus Z$, and each edge $uv$ in $G[A]$ has either the color of $u$ or the color of $v$. Of course, we let $Z = A$ and $G \setminus Z = B \cup T$.

\begin{subclaim}\label{Subclaim:Tourney}
Let $G$ be an edge-colored $K_n$ using $n$ colors and containing no properly colored $C_4$. If there is a subset $Z$ of $V(G)$ that has dependence property, then for distinct vertices $u$ and $v$, the color of $u$ and the color of $v$ are distinct.
\end{subclaim}
\begin{proof}
Suppose not, then $d^s(Z):=c(G)-c(G-Z)\leq|Z|-1$. Thus $c(G-Z)\geq c(G)-|Z|+1=|G-Z|+1$. It follows from Theorem~\ref{Thm:nP1C4} that $G-Z$ contains a properly colored $C_4$, a contradiction.
\end{proof}

Note that with the current structure, no vertices of $B$ can form a properly colored $C_{4}$ with any vertices of $G \setminus B$. This means that we will treat $B$ as a black box from now on. Also note that if $A = \emptyset$ (although we already know it is not), then $G$ would have Structure~(2).

Our goal is to show that, within $A$, there is either a vertex with all one color on its edges, yielding Structure~(1), or a rainbow triangle with appropriately colored edges to yield Structure~(2). With these two goals in mind, we may focus entirely within $A$ since we have already shown that either of these structures would extend from $A$ to all of $G$.

Each vertex in $A$ has an associated color, the color of its edges to $T$. We may therefore color the vertices of $A$ with their associated colors. Given two vertices $u, v \in A$, we have already shown that the color of $uv$ must be either the color of $u$ or the color of $v$. We may therefore orient all edges toward the vertex of the same color.
It follows from Subclaim~\ref{Subclaim:Tourney} that $D$ is a tournament since all colors must be distinct.

If $D$ has a sink, then this sink has all one color on all incident edges, so $G$ has Structure~(1).
This means that there is a directed cycle in $D$. Moreover, this means that there must be a directed triangle $U$ containing no two vertices of the same color.

If all edges from $D \setminus U$ to $U$ can be oriented to point toward $U$, then $A$ (and similarly $G$) has Structure~(2). Thus, suppose there exists a vertex $w$ that receives an arc directed away from $U$. To avoid producing a properly colored $C_{4}$, all other arcs must also be directed from $U$ to $w$. Thus, we may partition $D \setminus U$ into two vertex sets $IN$ and $OUT$, consisting of those vertices with all incident edges pointing into $U$ and those vertices with all incident edges pointing out of $U$ respectively.

Let $w_{1} \in IN$ and $w_{2} \in OUT$. Then to avoid a properly colored $C_{4}$, the edge $w_{1}w_{2}$ must be directed from $w_{1}$ to $w_{2}$. In general, all edges between $IN$ and $OUT$ must be directed from $IN$ to $OUT$. Since all edges of $D \setminus OUT$ point toward $OUT$, we may reduce the problem to a strictly smaller question of finding either Structure~(1) or Structure~(2) within $OUT$, to complete the proof of Claim~\ref{Claim:NonHead}.
\end{proof}

The end of the previous argument can be formalized in the following fact.

\begin{fact}\label{Fact:DP}
Let $G$ be an edge-colored $K_n$ using $n$ colors and containing no properly colored $C_4$. If there exists a set of vertices $A$ that can be colored so that all edges from a vertex $v\in A$ to $G\setminus A$ have the color of $v$ and all edges between two vertices $u,v\in A$ have either the color of $u$ or the color of $v$, then G has either Structure~(1) or Structure~(2).
\end{fact}

By Claim~\ref{Claim:NonHead} and since we have $n$ colors on the edges, there must be a vertex $u$ that is the head of at least two colors. Call this vertex a \emph{shared head} and the colors dominated by this shared head are called \emph{shared colors}.

\begin{claim}\label{Claim:OneStar}
All but at most one of the shared colors of a shared head must induce a star.
\end{claim}

\begin{proof}
Suppose not, so suppose $u$ is a shared head, say with shared colors red and blue, each of which inducing a subgraph that is not a star. This means that each spine, of red and blue, contains at least $2$ vertices and a tail. Let $r_{2}$ (and $b_{2}$) be a spine vertex (other than $u$) of the red (respectively blue) spines. Let $t_{r}$ (and $t_{b}$) be the tail vertex of the red (respectively blue) spine.

Next we show that $\{ u, r_{2}, b_{2}, t_{r}, t_{b}\}$ is a set with no repetition. Certainly $\{u, r_{2}, t_{r}\}$ are all distinct and $\{u, b_{2}, t_{b}\}$ are all distinct. If there is any overlap between these sets, say for example $r_{2} = t_{b}$, then the edge $ur_{2} = ut_{b}$ must be both red and blue by the definition of a spine, clearly a contradiction.

Next we show that $r_{2}b_{2}$ and $t_{r}t_{b}$ are both neither red nor blue. For a contradiction, suppose that, for example, $t_{r}t_{b}$ is red. Then by the definition of a spine, since $u$ has a red edge to every vertex with at least one incident red edge, the edge $ut_{b}$ must be red. This is a contradiction because $ut_{b}$ must be blue since $t_{b}$ is the tail of the blue spine.

With neither of $r_{2}b_{2}$ and $t_{r}t_{b}$ being red or blue, we obtain a properly colored $C_{4}$ on $r_{2}b_{2}t_{b}t_{r}r_{2}$, for a contradiction.
\end{proof}

\begin{claim}\label{Claim:SingleEdge}
All but at most one of the shared colors at a shared head must induce a single edge.
\end{claim}

\begin{proof}
Let $u$ be a shared head, say with shared colors red and blue. Without loss of generality, suppose $u$ has at least two edges in each of red and blue, say red to vertices $v$ and $w$, and blue to vertices $x$ and $y$. By Claim~\ref{Claim:OneStar}, one of these colors, say red, must induce a star. Since red induces a star, we know that the edge $vw$ is not red and $vw$ is certainly not blue since $u$ is the head of blue. Let green be the color of the edge $vw$.

In order to keep the $4$-cycle $uvwxu$ from being properly colored, the edge $wx$ must be either green or blue but $w$ cannot have any blue edges since $u$ is the head of blue and the edge $uw$ is already colored red. By symmetry, this means that all edges between $\{v, w\}$ and $\{x, y\}$ (the blue neighborhood of $u$) must be green.

We next show that the vertices in the red neighborhood of $u$ induce (in $G$) a complete graph in green. Otherwise, there is a non-green edge, say purple, adjacent to a green edge, say $vw$, within the red neighborhood of $u$. Without loss of generality, let $wz$ be this purple edge. By the previous argument, $w$ has all green edges to the blue neighborhood of $u$ but by symmetry, $\{w, z\}$ must have all purple edges to the blue neighborhood of $u$, clearly a contradiction.

With these two observations together, we see that all edges between the red neighborhood of $u$ and the blue neighborhood of $u$ must be green. Furthermore, if there was an edge, say $xy$, within the blue neighborhood of $u$ that is not either blue or green, then $uxyvu$ is a rainbow (and therefore properly colored) $C_{4}$. This means that all edges within the blue neighborhood of $u$ must be either blue or green.

Since we have considered at least $5$ vertices so far ($\{u, v, w, x, y\}$) and these vertices induce a subgraph of $G$ using only $3$ colors, there must be more vertices in $G$ that are not in the red or blue neighborhoods of $u$. Let $z$ be such a vertex. Then $z$ has a new color on its edge to $u$, say purple (see Figure \ref{Fig:RedBlueNeibor}).

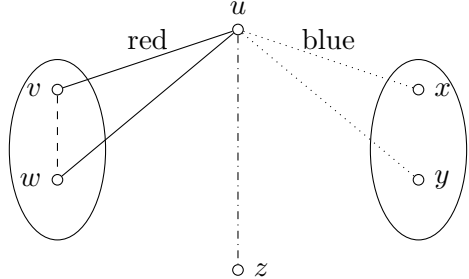
\begin{figure}[bht]
	\centering
	\begin{tikzpicture}[x=0.8cm, y=0.8cm]

	\vertex (v) at (0,3) [label=left:{$v$}]{};
	\vertex (w) at (0,1.5) [label=left:{$w$}]{};
	\vertex (u) at (3,4) [label=above:{$u$}]{};
	\vertex (z) at (3,0) [label=right:{$z$}]{};
	\vertex (x) at (6,3) [label=right:{$x$}]{};
	\vertex (y) at (6,1.5) [label=right:{$y$}]{};
	
	\path (u) edge node[above]{red} (v) edge (w);
	\path[dotted] (u) edge node[above]{blue} (x) edge (y);
	\path[dashed] (v) edge (w);
	\path[dashdotted] (u) edge (z);
	
	
	\draw (0,2) ellipse (0.8 and 1.5) node {};
	\draw (6,2) ellipse (0.8 and 1.5) node {};	
	\end{tikzpicture}
	\caption{The red and the blue neighborhood of $u$.}
	\label{Fig:RedBlueNeibor}
\end{figure}


By considering $zuvwz$, we see that $c(zw)$ is either purple or green. Also by considering $zwuyz$, we see that $c(zy)=c(zw)$ ($c(zy)$ cannot be blue since $u$ is the head of blue). By symmetry, the color of all the edges from $z$ to both the red and blue neighborhoods should be all purple or all green and these are the only two possible cases.

Let $z_1$ and $z_2$ be two vertices not in the blue or red neighborhood of $u$. By the above discussion we know that $c(z_1w)$ is either green or $c(z_1u)$. It follows from the fact that the cycle $z_1wuz_2z_1$ is not properly colored, that
$$
c(z_1z_2)\in \{c(z_1w),c(z_2u)\}\subseteq \{\text{green},c(z_1u),c(z_2u)\}.
$$

Finally, this means that $u$ along with its red and blue neighborhoods contain at least $5$ vertices but contribute only $3$ colors to $G$, and each additional vertex contributes at most one new color to $G$, meaning that $G$ is colored with at most $n - 2$ colors, a contradiction.
\end{proof}

Suppose $u$ is the shared head with lonely color red to lonely vertex $v$. Let blue be the other color for which $u$ is the head and let $x$ be a vertex in the blue neighborhood of $u$. Then $xv$ must be a third color, say green.

Suppose there exists a vertex $w$ with a new color, say purple, to $x$. Then $wv$ must also be purple since otherwise $wvuxw$ would be a properly colored $C_{4}$. Also $wxvuw$ would be a properly colored $C_{4}$ unless $wu$ is also purple.

Let $A$ be the set of vertices with colors other than blue or green to $x$. Color each vertex of $A$ with the color of its edge to $x$.
	
The remaining vertices in $V(G)-(\{u,v,x\}\cup A)$ have an edge to $x$ with color either green or blue. For a vertex $z \in G \setminus \{v\}$ with a green edge to $x$. Then to avoid a properly colored $C_{4}$ on $zxuvz$, the edge $zv$ must also be green ($zu$ is unknown). By the same argument, we can show that for a vertex $z' \in G \setminus \{u\}$ with a blue edge to $x$. Then $z'u$ is blue ($z'v$ is not blue).

Let $A^\prime$ be a set of vertices including $A$ constructed in the following way: if such a vertex as $z$ satisfying $zu$ is green, then consider this vertex as part of $A^\prime$ with associated color green (i.e. the color of $zx,zv,zu$ are all green). Thus the edges from each vertex in $A^\prime$ to $\{u,v,x\}$ have the same color.

We will show that $A^\prime$ has dependence property. In the following discussion, keep in mind that neither red nor blue edges are incident to vertices in $A^\prime$. Let $w,y$ be two vertices in $A^\prime$ (if two such vertices exist). To avoid a properly colored $C_4$ on $wyvuw$, we see that the edge $wy$ must have either the color of $w$ or the color of $y$. For each vertex $z$ with $zv$ green and $zu$ not green, it follows from the assumption that $zvuwz$ and $zuvwz$ are not properly colored that $zw$ must be the color of $w$. Similarly, for each vertex as $z^\prime$ with $z^\prime u$ blue and $z^\prime v$ not blue, it follows from the assumption that $zvuwz$ and $zuvwz$ are not properly colored that $z^\prime w$ must be the color of $w$.

If $A^\prime$ is not empty, by Fact 1, the graph $G$ has either Structure (1) or Structure (2), as required. We may therefore assume $A^\prime$ is empty.

Thus, each of the vertices in $G-\{u,v,x\}$ has one of two types, either as $z$ with $zx$, $zv$ green $zu$ not green, or as $z^\prime$ with $z^\prime x$, $z^\prime u$ blue, $z^\prime v$ not blue (See Figure \ref{Fig:Twotype}). Let
\begin{eqnarray*}
	Z&:=&\bigl\{w\in G-\{u,v,x\}~\big|~c(wx)=c(wv)=\text{green}, c(zu)\neq \text{green}\bigr\},\\
	Z^\prime &:=& \bigl\{w\in G-\{u,v,x\}~\big|~c(wx)=c(wu)=\text{blue}, c(wv)\neq \text{blue}\bigr\}.
\end{eqnarray*}

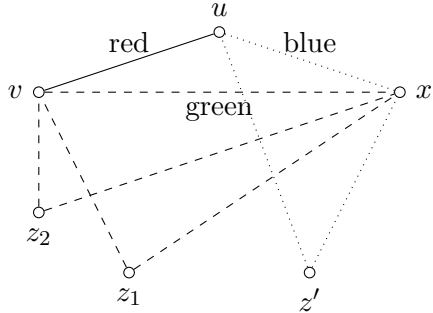
\begin{figure}[bht]
	\centering
	\begin{tikzpicture}[x=0.8cm, y=0.8cm]
	
	\vertex (v) at (0,3) [label=left:{$v$}]{};
	\vertex (u) at (3,4) [label=above:{$u$}]{};
	\vertex (x) at (6,3) [label=right:{$x$}]{};
	\vertex (z2) at (0,1) [label=below:{$z_2$}]{};
	\vertex (z1) at (1.5,0) [label=below:{$z_1$}]{};
	\vertex (z0) at (4.5,0) [label=below:{$z^\prime$}]{};
	
	\path (u) edge node[above]{red} (v);
	\path[dotted] (u) edge node[above]{blue} (x);
	\path[dotted] (z0) edge (u) edge (x);
	\path[dashed] (v) edge node[below]{green} (x);
	\path[dashed] (z1) edge (v) edge (x);
	\path[dashed] (z2) edge (v) edge (x);

	\end{tikzpicture}
	\caption{Two types of the vertices in $G-\{u,v,x\}$.}
	\label{Fig:Twotype}
\end{figure}

If there exist $z\in Z$ and $z^\prime\in Z^\prime$, then it follows from $zvuz^\prime z$ is not properly colored that
\[
c(z z^\prime)\subseteq \{c(z v),c(z^\prime u)\}=\{\text{green}, \text{blue}\}.
\]
Then $c(zu)$ and $c(z^\prime v)$ cannot both be colors other than blue or green (respectively). Otherwise $zuz^\prime vz$ is a properly colored $C_4$. Thus $G[\{u,v,x,z,z^\prime\}]$ contains at most $4$ colors and each other vertex contributes at most one new color to $G$. Thus $G$ contains at most $n-1$ colors, a contradiction.

So either $Z=\emptyset$ or $Z'=\emptyset$. By symmetry, we assume that $Z'=\emptyset$. For two vertices $z_1,z_2\in Z$ (if they exist), it follows from $z_1 uvz_2z_1$ and $z_1 xuz_2z_1$ are not properly colored that
\[
	c(z_1z_2)\subseteq \{c(z_1u),\text{green}\}\cap \{c(z_2u),\text{green}\}.
\]
Moreover, $G[\{u,v,x\}]$ has $3$ colors and each other vertex contributes at most one more color to $G$. In order to have $n$ colors in $G$, each other vertex has to contribute exactly one more color to $G$. There follows
\[
	c(z_1u)\neq c(z_2u)\quad \text{and}\quad c(z_1z_2)=\text{green}.
\]
Which means $G$ has Structure (3) in which $u$ is the center of the rainbow star, as required. (For the case there are only vertices as $z^\prime$, the center of the rainbow star is $v$.)
\end{proof}

\section{Proof of Theorem~\ref{thm_C4inG}}

\begin{proof}
We prove this result by induction on $n$. It is easy to check when $n=4$. Assume the conclusions are true for all graphs with order less than $n$ where $n \geq 5$.
	
This theorem has two parts. We prove the latter part, so then the former part follows in the following way: For $G$ with $e(G)+c(G)\geq \frac{n(n+1)}{2}+1$, if $G$ is not rainbow, we can delete edges with repeating colors one by one, until either we get a subgraph $G_0$ with $e(G_0)+c(G_0)=\frac{n(n+1)}{2}$ (then there is a properly colored $C_4$ by the latter conclusion since $G_0$ is not complete) or we get a rainbow subgraph $G_0$ with $e(G_0)\geq \frac{n(n+1)}{4}+\frac{1}{2}$. For $n=5$, we have
$$
e(G_0)\geq\frac{n(n+1)}{4}+\frac{1}{2}=8.
$$
Easy to find a $C_4$ (which is rainbow) in $G_0$. For $n\geq 6$, by Lemma \ref{lem_Kovari1954}, we get
$$
e(G_0)>\frac{n(n+1)}{4}>\frac{1}{2}n^{3/2}+\frac{1}{2}n\geq ex(n,C_4).
$$
Thus $G_0$ contains a $C_4$ which is rainbow (also properly colored).

Now we prove the latter part of this theorem. Let $G$ be an edge-colored graph with order $n$, satisfying $e(G)+c(G)=\frac{n(n+1)}{2}$ and there is no properly colored $C_4$ in $G$.

\setcounter{claim}{0}

\begin{claim}
$G$ is not rainbow.
\end{claim}

\begin{proof}
It is trivial when $e(G)+c(G)$ is odd, for example when $n=5$. When $e(G)+c(G)$ is even ($n\geq 6$), suppose $G$ is rainbow. It follows from Lemma \ref{lem_Kovari1954} that
$$
e(G)=\frac{n(n+1)}{4}>\frac{1}{2}n^{3/2}+\frac{1}{2}n\geq ex(n,C_4).
$$
Thus $G$ contains a $C_4$ which is rainbow (also properly colored), a contradiction.
\end{proof}

\begin{claim}
For all $v\in V(G)$, $d(v)+d^s(v)\geq n$.
\end{claim}

\begin{proof}
Suppose not, there is a vertex $u$ satisfying $d(u)+d^s(u)\leq n-1$. Then
\beqs
e(G-u)+c(G-u) & = & e(G)-d(u)+c(G)-d^s(u)\\
~ & \geq & \frac{n(n+1)}{2}-(n-1)\\
~ & = & \frac{n(n-1)}{2}+1.
\eeqs
	
By the induction hypothesis, $G-u$ contains a properly colored $C_4$, a contradiction.
\end{proof}

\begin{claim}\label{Claim:3}
There exists a vertex $u\in V(G)$, such that $d(u)+d^s(u)=n$.
\end{claim}
	
\begin{proof}
Suppose not, i.e., for all $v\in V(G)$, $d(v)+d^s(v)\geq n+1$. By Lemma \ref{lem_CSdegree}, we have $\sum_v d^s(v)< 2c(G)$ since $G$ is not rainbow. Thus
$$
n(n+1)\leq\sum_v \bigl[d(v)+d^s(v)\bigr]<2e(G)+2c(G)=n(n+1),
$$
a contradiction.
\end{proof}
	
Let $u$ be the vertex guaranteed by Claim~\ref{Claim:3}. It is easy to see that $e(G-u)+c(G-u)=\frac{n(n-1)}{2}$. By the induction hypothesis, $G-u$ is complete. By  Theorem~\ref{thm_C4inKnstruc}, $G-u$ has a layered structure (see Figure \ref{Fig:C4inKnstruc}). Let $s$ be the center and $l_1,l_2$ be two leaves of the rainbow star in the center part (third part in Theorem \ref{thm_C4inKnstruc}) of this layered structure.

Suppose $G$ is not complete, i.e. $d(u)\leq n-2$. Then $d^s(u)\geq 2$. Let $uv_1$, $uv_2$ be two edges with distinct colors saturated by $u$. If $s\not\in \{v_1,v_2\}$, then $uv_1sv_2u$ is a rainbow $C_4$, a contradiction. If $s\in \{v_1,v_2\}$, say $v_1=s$, then $uv_2l_isu$ is a rainbow $C_4$, in which $l_i\neq v_2$ for some $i\in\{1,2\}$, a contradiction.
	
The proof is complete.
\end{proof}

\section{Proof of Theorem~\ref{Thm:DegSumforC4}}\label{Sec:relation}

\begin{theorem}\label{thm_starforest}
	If each color class of an edge-colored graph $G$ is a star forest, then $e(G)+c(G)\leq \sum_v d^c(v)$. The equality holds if (not if and only if) each color class of $G$ is a star.
\end{theorem}

\begin{proof}
	Let $G_1$ be the edge-colored graph constructed from $G$ by assigning distinct colors to each component of those color classes. There holds $d^c_{G_1}(v)=d^c_{G}(v)$, for all $v\in V(G)$. Let $D$ be an orientation of $G_1$ such that all those edges of each star (color classes in $G_1$) point from the center to leaf vertices. If the color class is a single edge, choose one of the two possible orientations arbitrarily. See Figure \ref{Fig:Evolution}  for example.
	
	\begin{figure}[bht]
		\centering
		\begin{tikzpicture}[x=0.8cm, y=0.8cm]
		\vertex (y1) at (4,2.5){}; \vertex (y2) at (4,1.5){};
		\vertex (y3) at (5,2){}; \vertex (y4) at (6,2){};
		\vertex (y5) at (6,3){}; \vertex (y6) at (5,1){};
		\vertex (y7) at (4,0){}; \vertex (y8) at (5,0){};
		\vertex (y9) at (6,0){};
		
		\vertex (z1) at (8,2.5){}; \vertex (z2) at (8,1.5){};
		\vertex (z3) at (9,2){}; \vertex (z4) at (10,2){};
		\vertex (z5) at (10,3){}; \vertex (z6) at (9,1){};
		\vertex (z7) at (8,0){}; \vertex (z8) at (9,0){};
		\vertex (z9) at (10,0){};
		
		\vertex (d1) at (12,2.5){}; \vertex (d2) at (12,1.5){};
		\vertex (d3) at (13,2){}; \vertex (d4) at (14,2){};
		\vertex (d5) at (14,3){}; \vertex (d6) at (13,1){};
		\vertex (d7) at (12,0){}; \vertex (d8) at (13,0){};
		\vertex (d9) at (14,0){};
		
		\path (y3) edge (y1) edge (y2);
		\path (y6) edge (y7) edge (y8) edge (y9);
		\path (y4) edge (y5);
		
		\path (z3) edge (z1) edge (z2);
		\path[dashed] (z6) edge (z7) edge (z8) edge (z9);
		\path[dotted] (z4) edge (z5);
		
		\path[->] (d3) edge (d1) edge (d2);
		\path[dashed,->] (d6) edge (d7) edge (d8) edge (d9);
		\path[dotted,->] (d4) edge (d5);
		
		\node (p2) at (5,-1) [label=center:{$G^i$ in $G$}]{};
		\node (p3) at (9,-1) [label=center:{in $G_1$}]{};
		\node (p4) at (13,-1) [label=center:{in $D$}]{};	
		\end{tikzpicture}
		\caption{An example of the evolution of a color class.}
		\label{Fig:Evolution}
	\end{figure}
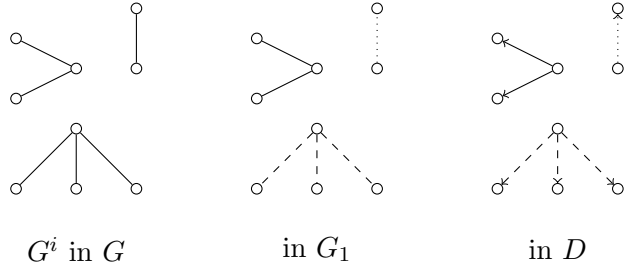
	
	Then for all $v\in V(G)$, $d^c_{G_1}(v)=d^-_D(v)+d^{c+}_D(v)$, where $d^{c+}_D(v)$ is the number of colors on the out-going arcs of $v$. Thus
	\beqs
	\sum_v d_G^c(v) = \sum_v d_{G_1}^c(v)  & = & \sum_v \bigl[d^-_D(v)+d^{c+}_D(v)\bigr]\\
	~ & = & e(G_1)+\sum_v d^{c+}_D(v)\\
	~ & = & e(G_1)+c(G_1)\\
	~ & \geq & e(G)+c(G).
	\eeqs
	The equality holds if $c(G_1)=c(G)$, i.e. each color class of $G$ is a star.
\end{proof}

\begin{theorem}\label{thm_relation}
	Let $f(n)$ be a function on $n$. If statement A: `For each edge-colored graph $G$ with order $n$, $e(G)+c(G)\geq f(n)$ is sufficient to force $G$ to contain a properly colored subgraph $H$.' is true, then statement B: `For each edge-colored graph $G$ with order $n$, $\sum_v d^c(v)\geq f(n)$ is sufficient to force $G$ to contain a properly colored subgraph $H$.' is also true.
\end{theorem}

\begin{proof}
	Let $f(n)$ and $H$ be given. Suppose statement $A$ is true, and statement $B$ is not true. Let $G$ be a counterexample of statement $B$ first with minimum order and then with minimum number of edges. It is easy to see that each color class of $G$ is a star forest. Let $G_1$ be the edge-colored graph constructed from $G$ by assigning distinct colors to each component of those color classes. Then we get $d^c_{G_1}(v)=d^c_{G}(v)$, for all $v\in V(G)$. Moreover, $G$ contains a properly colored $H$ if and only if $G_1$ contains a properly colored $H$. It follows from Theorem \ref{thm_starforest} that $e(G_1)+c(G_1)=\sum_v d_{G_1}^c(v)=\sum_v d_G^c(v)\geq f(n)$. Thus $G_1$ contains a properly colored $H$ by statement A, a contradiction.
\end{proof}

\begin{proof}[Proof of Theorem~\ref{Thm:DegSumforC4}]
It follows from Theorem \ref{thm_C4inG} and Theorem \ref{thm_relation} that: for an edge-colored graph $G$ order $n$ with $n \geq 4$, if $\sum_v d^c(v)\geq \frac{n(n+1)}{2}+1$, then $G$ contains a properly colored $C_4$. Moreover, if $\sum_v d^c(v) = \frac{n(n+1)}{2}$ and $G$ contains no properly colored $C_4$. Let $G_0$ be a minimum subgraph of $G$ obtained by edge deleting which maintains the color degree of each vertex. It is easy to see that each color class of $G_0$ is a star forest. Let $G_1$ be the edge-colored graph constructed from $G$ by assigning distinct colors to each component of those color classes. Then $G_1$ contains no properly colored $C_4$ and
\[
	e(G_1)+c(G_1)=\sum_v d_{G_1}^c(v)=\sum_v d_G^c(v)= \frac{n(n+1)}{2}.
\]
Thus $G_1$ has the layered structure by Theorem \ref{thm_C4inG} with the center part a rainbow triangle (in which each color class is a star). Since no two color classes are disjoint in $G_1$, we can deduce that $G\cong G_1$.
\end{proof}

By noting that the minimum color degree of those extremal graphs of Theorem \ref{Thm:DegSumforC4} is at most 2, we can deduce the following result.

\begin{corollary}
	Let $G$ be an edge-colored graph with order $n$ with $n\geq 4$. If $d^c(v)\geq (n+1)/2$ for all $v\in V(G)$, then $G$ contains a properly colored $C_4$.
\end{corollary}

Li et al.\cite{LNXZ2014} considered the color degree sum condition for the existence of rainbow triangles. They showed that the lower bound in the following theorem is sharp.

\begin{theorem}[Li, Ning, Xu, Zhang\cite{LNXZ2014}]
	Let G be an edge-colored graph on $n \geq 3$ vertices. If $\sum_v d^c(v)\geq n(n+1)/2$, then $G$ contains a rainbow $K_3$.
\end{theorem}

It follows from Theorem \ref{Thm:C3inG} and Theorem \ref{Thm:C3inGstruc} that: if $e(G)+c(G) \geq n(n+1)/2$, then $G$ contains a rainbow $K_3$; moreover, if $e(G)+c(G)=n(n+1)/2-1$ and $G$ contains no rainbow triangle, then $G$ is a Gallai colored complete graph with $n-1$ colors. By following a similar method used in the proof of Theorem~\ref{Thm:DegSumforC4}, it is easy to show that: if $\sum_v d^c(v)=n(n+1)/2-1$ and $G$ contains no rainbow triangle, then $G$ is a Gallai colored complete graph with $n-1$ colors and each color class is a star. Together with Theorem \ref{Thm:C3inGstruc}, we have the following result.

\begin{theorem}
If $\sum_v d^c(v)=n(n+1)/2-1$ and $G$ contains no rainbow triangle, then $G$ is a complete graph. Moreover, its vertices and colors can be relabeled so that $V(G)=\{v_1,v_2,\cdots,v_n\}$, where $c(v_iv_j)=i$, for all $1\leq i<j\leq n$.
\end{theorem}

\section{Classification}

\begin{lemma}\label{Lemma:Oneortwo}
	Let $G$ be an edge-colored complete graph. If each color class and each pair of color classes form a threshold graph, then $G$ contains no properly colored $C_4$.
\end{lemma}

\begin{proof}
	Suppose not, there is a properly colored $C_4=v_1v_2v_3v_4v_1$ in $G$ in which $c(v_1v_2)=i$ and $c(v_3v_4)=j$. Let $H=G^i\cup G^j$, which is a threshold graph by the condition. Thus $v_1v_2$, $v_3v_4$ are in $H$ and $v_2v_3$, $v_4v_1$ are not in $H$. This implies $H[\{v_1,v_2,v_3,v_4\}]$ is an induced $2K_2 $, $P_4$ or $C_4$, a contradiction to Lemma~\ref{Lemma:Forbid}.
\end{proof}

Note that the complement of a threshold graph is also a threshold graph. We can deduce the following corollary from Lemma \ref{Lemma:Oneortwo}.
\begin{corollary}
	Let $G$ be a $2$-colored threshold graph in which each color class is
	a threshold graph. Then $G$ contains no properly colored $C_4$.
\end{corollary}

The following theorem follows from Lemma \ref{Lemma:Dominating} and Lemma \ref{Lemma:Oneortwo}.
\begin{theorem}
	Let $G$ be an edge-colored $K_n$, then $G$ contains no properly colored $C_4$ if and only if
	each set of $t$ color classes form a threshold graph for $1 \leq t \leq 2$.
\end{theorem}

In the following we discuss the structure of 2-colored and 3-colored complete graphs which contains no properly colored $C_4$.

Given a subgraph $G^{i}$ with spine $S_{i}$, we call the set of vertices in $G^{i} \setminus S_{i}$ the \emph{ribs} of $G^{i}$ (see Figure~\ref{Fig:Threshold}), denoted by $R_{i}$. As the spine vertices are removed from $G^{i}$ in order, the ribs become isolated and are also removed in some order. We therefore refer to the ribs as an ordered set of vertices $(u_{1}, u_{2}, \dots, u_{r})$, although just like the spine, this ordering need not be unique.

A {\em drawing} of a threshold graph is a partition $(S_{i},R_{i})$ of its vertices such that $V_{i}$ is the ordered set of spine vertices, and $R_{i}$ is the ordered set of ribs. If each color class of $G$ is a threshold graph, then a {\em drawing} of $G$ is a sequence of the drawings of each color class.

\begin{lemma}\label{lem_2cKn}
	Let $G$ be a 2-colored $K_n$ where each color class is a threshold graph. Then there is a drawing of $G$ such that $V(G)=S_1\cup S_2\cup \{u\}$, where $S_i$ is the spine of $G^i$ and $u$ is a tail of both spines.
\end{lemma}

\begin{proof}
	Consider a drawing of $G^1$ with ordered spine $S_1$ and ordered ribs $R_1$. Let $u$ be the last tail in the ribs.
	See Figure~\ref{Fig:Threshold} for example.
	
	Within $G^1$, by the definition, for each $v_i\in R_1\setminus \{u\}$, we have that $v_{i}$ is not adjacent in color 1 to each $u_k\in S_1$ for all $k$ greater than some $j$. In other words, for each $v_i\in R_1\setminus\{u\}$ (in $G^1$), we have that $v_{i}$ is adjacent in color 2 to each $u_k\in S_1$ for all $k$ greater than some $j$. Moreover, each $v_i\in R_1\setminus\{u\}$ is adjacent to $u$ in color $2$. Thus we may choose $S_2=R_1\setminus\{u\}$ to be a spine of $G^2$ (with the inherited ordering), and $R_2=S_1\cup\{u\}$ to be the set of ribs of color 2 (with the inherited ordering).
\end{proof}

\begin{lemma}\label{lem_idTG}
	Let $G$ be a threshold graph with drawing $(V, R)$. For $V' \subseteq V$ and $R' \subseteq R$, $G[V' ,R']$ is also a threshold graph.
\end{lemma}

\begin{proof}
	Given a threshold graph $G$ with drawing $(V, R)$, let $V' \subseteq V$ and $R' \subseteq R$. Then $V'$ and $R'$, with the orderings inherited from $V$ and $R$ respectively, form a drawing $(V', R')$ of a (smaller) threshold graph.
\end{proof}

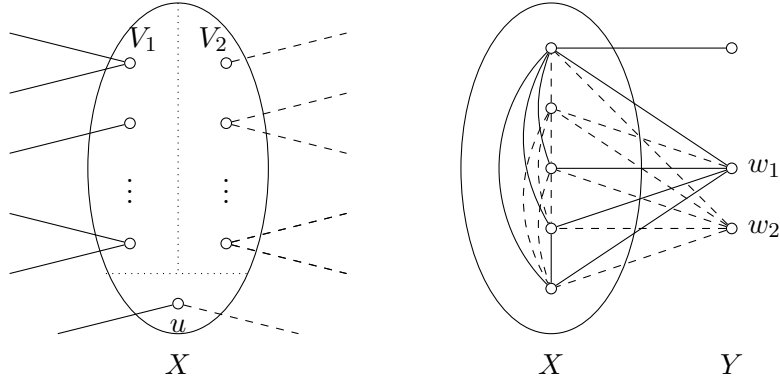
\begin{figure}[bht]
	\centering
	\begin{tikzpicture}[x=0.8cm, y=0.8cm]
	\vertex (v1) at (0,3) []{}; \vertex (v2) at (0,2) []{};
	\vertex (v3) at (0,0) []{};
	\vertex (w1) at (1.6,3) []{}; \vertex (w2) at (1.6,2) []{};
	\vertex (w3) at (1.6,0) []{};
	\vertex (u1) at (0.8,-1) [label=below:{$u$}]{};
	
	\path (v1) edge (-2,2.5) edge (-2,3.5);
	\path (v2) edge (-2,1.5);
	\path (v3) edge (-2,0.5) edge (-2,-0.5);
	\path[dashed] (w1) edge (3.6,3.5);
	\path[dashed] (w2) edge (3.6,1.5) edge (3.6,2.5);
	\path[dashed] (w3) edge (3.6,0.5) edge (3.6,-0.5);
	\path[dashed] (w3) edge (3.6,0.5) edge (3.6,-0.5);
	\path[dashed] (u1) edge (2.8,-1.5); \path (u1) edge (-1.2,-1.5);
	
	\path[dotted] (0.8,4) edge (0.8,-0.5);
	\path[dotted] (-0.4,-0.5) edge (2,-0.5);
	\draw (0.8,1.25) ellipse (1.5 and 2.75) node {};
	
	\node (p1) at (0.8,3.4) [label=left:{$V_1$}]{};
	\node (p2) at (0.8,3.4) [label=right:{$V_2$}]{};
	\node (p3) at (0,1) [label=center:{$\vdots$}]{};
	\node (p4) at (1.6,1) [label=center:{$\vdots$}]{};
	
	\draw (7,1.25) ellipse (1.5 and 2.75) node {};
	
	\vertex (w1) at (7,3.25) []{}; \vertex (w2) at (7,2.25) []{};
	\vertex (w3) at (7,1.25) []{}; \vertex (w4) at (7,0.25) []{};
	\vertex (w5) at (7,-0.75) []{};
	\vertex (y1) at (10,0.25) [label=right:{$w_2$}]{};
	\vertex (y2) at (10,1.25) [label=right:{$w_1$}]{};
	\vertex (y3) at (10,3.25) []{};
	
	\path (y2) edge (w1) edge (w3) edge (w4) edge (w5);
	\path[dashed] (y2) edge (w2); \path (y3) edge (w1);
	\path[dashed] (y1) edge (w1) edge (w2) edge (w3) edge (w4) edge (w5);
	\path (w1) edge[bend right=20] (w3) edge[bend right=30] (w4) edge[bend right=45] (w5);
	\path[dashed] (w2) edge (w1) edge (w3) edge[bend right=20] (w4) edge[bend right=30] (w5);
	\path[dashed] (w3) edge (w4) edge[bend right=20] (w5);
	\path (w4) edge (w5);
	
	\node (p5) at (0.8,-2) [label=center:{$X$}]{};
	\node (p6) at (7,-2) [label=center:{$X$}]{};
	\node (p7) at (10,-2) [label=center:{$Y$}]{};
	\end{tikzpicture}
	\caption{Two possible structures of a 2-colored threshold graph.}
\end{figure}

\begin{theorem}\label{thm_2cTG}
	Let $G$ be a 2-colored threshold graph in which each color class is also a threshold graph. Assume $G$ has a drawing $(X,Y)$. Then either
	\begin{itemize}
		\item both the spines of these color classes belong to $X$; then the spine of $G$ has a partition $V_1\cup V_2\cup\{u\}$, such that for $i\in \{1,2\}$, $V_i$ or $V_i\cup \{u\}$ is a spine of $G^i$, and $[v,Y]$ is monochromatic with color $i$ for all $v\in V_i$; or
		\item at least one color class has a spine vertex in $Y$, denote by $w_1$ and $w_2$ respectively if they exist; then $G-\{w_1,w_2\}$ has the first structure.
	\end{itemize}
	
\end{theorem}

\begin{proof}
	Let $G[X]$ be the subgraph of $G$ induced on the spine of $G$, so it is a 2-colored complete graph. By Lemma~\ref{Lemma:Dominating}, both $G^1[X]$ and $G^2[X]$ are threshold graphs. It follows from Lemma \ref{lem_2cKn} that there is a drawing of $G[X]$ such that $X=V_1\cup V_2\cup \{u\}$ where $V_{i}$ is the spine of color $i$ (within $X$) and $u$ is a shared tail of both colors (again restricted to $X$).
	
	If $X$ contains only one vertex, then $G$ is a star and we can set $V_1=V_2=\emptyset$ and $G$ has the first structure. So we can assume that $X$ contains at least two vertices.
	
	\setcounter{case}{0}
	\begin{case}
		Both the spines of these color classes belong to $X$.
	\end{case}
	
	Since $V_{i} \cup \{u\}$ is complete in color $i$ and $Y$ contains no spine vertices of either color, there is at most one vertex $v_{i} \in V_{i} \cup \{u\}$ with any edges of color $3 - i$ to $Y$ since such a vertex must be in the spine of color $3 - i$, for $\in \{1,2\}$. If both $v_{1}$ and $v_{2}$ exist and are distinct, then without loss of generality, we may assume $c(v_{1}v_{2}) = 1$. Note that $v_{1} \neq u$ since $c(v_{2}u) = 2$. Since $v_{1}$ is in the spine of color $2$ and $c(uv_{1}) = 1$, $u$ must not be in the spine of color $2$. This means that all edges from $u$ to $Y$ must have color $1$. Then $u$ is in the spine of color $1$, which means $v_{2} = u$ since otherwise $c(uv_{2}) = 2$ so both cannot be in the spine of color $1$. Then $V_{1} - v_{1} + u\cup V_{2}\cup v_{1}$ is the desired partition, such that $V_{1} - v_{1} + u$ is a spine of color $1$ and $V_{2} + v_{1}$ is a spine of color $2$, completing the proof of this case.
	
	\begin{case}
		At least one color class having a spine vertex in $Y$.
	\end{case}
	
	Since $Y$ is an independent set, if $G^i$ has a spine vertex in $Y$, there must be only one such vertex. Denote it by $w_i$ if it exists. By Lemma~\ref{Lemma:Dominating}, both $G^i-w_i$ and $G-\{w_1,w_2\}$ are threshold graphs. Moreover, the spine of $G^i-w_i$ is in the spine of $G-\{w_1,w_2\}$. It follows from Case 1 that there is a partition as required.
\end{proof}

\section*{Acknowledgments}
The first author's research was partially supported by NSFC (Nos.\ 11701441, 11601429 and 11671320). The third author's research was partially supported by NSFC (Nos.\ 11671320 and U1803263)  and the Fundamental Research Funds for the Central Universities (No. 3102019GHJD003).


\end{document}